\documentclass[12]{article}
\usepackage{amsmath,amssymb,amsthm,color}
\usepackage{times,fancyhdr}
\usepackage{graphicx}
\usepackage{geometry}
\usepackage{titlesec}
\usepackage{cite}

\usepackage{amssymb}
\usepackage{amsmath}
\usepackage{amsfonts}
\usepackage{amsthm,amscd}
\usepackage{latexsym}
\usepackage{comment}

\newtheorem{thm}{Theorem}[section]

\newtheorem{prop}{Proposition}[section]
\newtheorem{cor}{Corollary}[section]
\newtheorem{defn}{Definition}[section]

\def\R{{\mathfrak R}\, }
\def\M{{\mathfrak M}\, }

\def\G{{\mathfrak G}\, }
\def\Z{{\mathfrak Z}\, }
\def\ci{\begin{color}{red}\,}
\def\cf{\end{color}\,}

\begin{document}

\begin{center}{\bf \LARGE Multiplicative Maps on Generalized $n$-matrix Rings}\\
\vspace{.2in} 
 
{\bf Bruno L. M. Ferreira}\\
{\it Federal University of Technology,\\
Professora Laura Pacheco Bastos Avenue, 800,\\
85053-510, Guarapuava, Brazil.}\\
e-mail: brunoferreira@utfpr.edu.br\\
and\\
{\bf Aisha Jabeen}\\
{\it Department  of Applied Sciences \& Humanities,\\ Jamia Millia Islamia,\\ New Delhi-110025, India.}\\
e-mail: ajabeen329@gmail.com\\

\end{center}

\begin{abstract}
Let $\mathfrak{R}$ and $\mathfrak{R}'$ be two associative rings (not necessarily with the identity elements).  A bijective map $\varphi$ of $\mathfrak{R}$ onto  $\mathfrak{R}'$ is called a \textit{$m$-multiplicative isomorphism} if
{$\varphi (x_{1} \cdots x_{m}) = \varphi(x_{1}) \cdots \varphi(x_{m})$}
for all $x_{1}, \cdots ,x_{m}\in \mathfrak{R}.$
In this article, we establish a condition on generalized $n$-matrix rings, that assures that multiplicative maps are additive on generalized $n$-matrix rings under certain restrictions. And then, we apply our result for study of $m$-multiplicative isomorphism and  $m$-multiplicative derivation on generalized $n$-matrix rings.
\end{abstract}


\section{Introduction}
Let $\mathfrak{R}$ and $\mathfrak{R}'$ be two associative rings (not necessarily with the identity elements). We denote by $\mathfrak{Z}(\mathfrak{R})$ the center of $\mathfrak{R}.$ A bijective map $\varphi$ of $\mathfrak{R}$ onto  $\mathfrak{R}'$ is called a \textit{$m$-multiplicative isomorphism} if\\
\centerline{$\varphi (x_{1} \cdots x_{m}) = \varphi(x_{1}) \cdots \varphi(x_{m})$}\\
for all $x_{1}, \cdots ,x_{m}\in \mathfrak{R}.$ In particular, if $m = 2$ then $\varphi$ is called a \textit{multiplicative isomorphism}. Similarly, a map $d$ of $\mathfrak{R}$ is called a \textit{$m$-multiplicative derivation} if\\
\centerline{$d(x_{1} \cdots x_{m}) = \sum _{i=1}^{m} x_{1} \cdots d(x_{i}) \cdots x_{m}$}\\
for all $x_{1}, \cdots ,x_{m}\in \mathfrak{R}.$ If $d(xy)=d(x)y + xd(y)$ for all $x, y\in \mathfrak{R}$, we just say that $d$ is a {\it multiplicative derivation} of $\mathfrak{R}$.
\par
In last few decades, the multiplicative mappings on rings and algebras has been studied by many authors \cite{Mart, Wang, Lu02, LuXie06, ChengJing08, LiXiao11}. Martindale \cite{Mart} established a condition on a ring such that multiplicative bijective mappings on this ring are all additive. In particular, every multiplicative bijective mapping from a prime ring containing a nontrivial idempotent onto an arbitrary ring is additive. Lu \cite{Lu02} studied multiplicative isomorphisms of subalgebras of nest algebras which contain all finite rank operators but might contain no idempotents and proved that these multiplicative mappings are automatically additive and linear or conjugate linear. Further, Wang in \cite{Wangc, Wang} considered the additivity of multiplicative maps on rings with idempotents and triangular rings respectively. Recently, in order to generalize the result in \cite{Wang} first author \cite{Ferreira},  defined a class of ring called triangular
$n$-matrix ring and studied the additivity of multiplicative maps on that class of rings. In view of above discussed literature, in this article we discuss the additivity of multiplicative maps on a more general class of rings called generalized $n$-matrix rings.
\par
We adopt and follow the same structure of the article and demonstration presented in \cite{Ferreira}, in order to preserve the author ideas and to highlight the generalization of the triangular $n$-matrix results to the generalized $n$-matrix results.
\begin{defn}\label{pri}
Let $\R_1, \R_2, \cdots, \R_n$ be rings and $\M_{ij}$ $(\R_i, \R_j)$-bimodules with $\M_{ii} = \R_i$ for all $i, j \in \left\{1, \ldots, n\right\}$. Let $\varphi_{ijk}: \M_{ij} \otimes_{\R_j} \M_{jk} \longrightarrow \M_{ik}$ be $(\R_i, \R_k)$-bimodules homomorphisms with $\varphi_{iij}: \R_i \otimes_{\R_i} \M_{ij} \longrightarrow \M_{ij}$ and $\varphi_{ijj}: \M_{ij} \otimes_{\R_j} \R_j \longrightarrow \M_{ij}$ the canonical isomorphisms for all $i, j, k \in \left\{1, \ldots, n\right\}$. Write $a \circ b = \varphi_{ijk}(a \otimes b)$ for $a \in \M_{ij},$ $b \in \M_{jk}.$ We consider
\begin{enumerate}
\item[{\it (i)}] $\M_{ij}$ is faithful as a left $\R_i$-module and faithful as a right $\R_j$-module with $i\neq j,$
\item[{\it (ii)}] if $m_{ij} \in \M_{ij}$ is such that $\R_i m_{ij} \R_j = 0$ then $m_{ij} = 0$  with $i\neq j.$
\end{enumerate}
Let \begin{eqnarray*} \G = \left\{\left(
\begin{array}{cccc}
r_{11} & m_{12} & \ldots & m_{1n}\\
m_{21}& r_{22} & \ldots & m_{2n}\\
\vdots & \vdots  & \ddots & \vdots\\
m_{n1} & m_{n2}  & \ldots & r_{nn}\\
\end{array}
\right)_{n \times n}~ : ~\underbrace{ r_{ii} \in \R_{i} ~(= \M_{ii}), ~ m_{ij} \in \M_{ij}}_{(i, j \in \left\{1, \ldots, n\right\})}
\right\}\end{eqnarray*}
be the set of all $n \times n$ matrices $[m_{ij}]$ with the $(i, j)$-entry $m_{ij} \in \M_{ij}$ for all $i,j \in \left\{1, \ldots , n\right\}$. Observe that, with the obvious matrix operations of addition and multiplication, $\G$ is a ring iff $a \circ (b \circ c) = (a \circ b) \circ c$ for all $a \in \M_{ik}$, $b \in \M_{kl}$ and $c \in \M_{lj}$ for all $i, j, k, l \in \left\{1, \ldots, n\right\}$. When $\G$ is a ring, it is called a \textit{generalized $n-$matrix ring}.
\end{defn}
Note that if $n = 2,$ then we have the generalized matrix ring. We denote by $ \bigoplus^{n}_{i = 1} r_{ii}$ the element
$$\left(\begin{array}{cccc}
r_{11} &  &  & \\
 & r_{22} &  & \\
 &  & \ddots & \\
 &  &  & r_{nn}\\
\end{array}\right)$$
in $\G.$

\pagestyle{fancy}
\fancyhead{}
\fancyhead[EC]{B. L. M. Ferreira}
\fancyhead[EL,OR]{\thepage}
\fancyhead[OC]{Multiplicative Maps on Generalized $n$-matrix Rings}
\fancyfoot{}
\renewcommand\headrulewidth{0.5pt}

Set $\G_{ij}= \left\{\left(m_{kt}\right):~ m_{kt} = \left\{{ \begin{matrix} m_{ij}, & \textrm{if}~(k,t)=(i,j)\\ 0, & \textrm{if}~(k,t)\neq (i,j)\end{matrix}}, ~i, j \in \left\{1, \ldots, n\right\} \right\}.\right.$ Then we can write $\displaystyle  \G = \bigoplus_{ i, j \in \left\{1, \ldots , n\right\}}\G_{ij}.$ Henceforth the element $a_{ij}$ belongs $\G_{ij}$ and the corresponding elements are in $\R_1, \cdots, \R_n$ or $\M_{ij}.$ By a direct calculation $a_{ij}a_{kl} = 0$ if $j \neq k.$
We define natural projections $\pi_{\R_{i}} : \G \longrightarrow \R_{i}$ $(1\leq i\leq n)$ by  $$\left(\begin{array}{cccc}
r_{11} & m_{12} & \ldots & m_{1n}\\
m_{21} & r_{22} & \ldots & m_{2n}\\
 \vdots & \vdots & \ddots & \vdots\\
 m_{n1 }& m_{n2}  & \ddots & r_{nn}\\
\end{array}\right)\longmapsto r_{ii}.$$
The following result is a characterization of center of generalized $n$-matrix ring.
\begin{prop}\label{seg}
Let $\G$ be a generalized $n-$matrix ring. The center of $\G$ is \\
\centerline{$\mathfrak{Z}(\G) = \left\{ \bigoplus_{i=1}^{n} r_{ii} ~\Big|~ r_{ii}m_{ij} = m_{ij}r_{jj} \mbox{ for all }  m_{ij} \in \M_{ij}, ~i \neq j\right\}.$}\\
Furthermore, $\mathfrak{\Z}(\G)_{ii} \cong \pi_{\R_i}(\mathfrak{Z}(\G))\subseteq \mathfrak{\Z}(\R_i)$, and there exists a unique ring
isomorphism $\tau^j_{i}$ from $\pi_{\R_i}(\Z(\G))$ to $\pi_{\R_j}(\Z(\G))$ $i \neq j$ such that $r_{ii}m_{ij} = m_{ij}\tau^j_{i}(r_{ii})$ for all $m_{ij} \in \M_{ij}.$
\end{prop}
\begin{proof} Let $S = \left\{ \bigoplus_{i=1}^{n} r_{ii} ~\Big|~ r_{ii}m_{ij} = m_{ij}r_{jj} \mbox{ for all }  m_{ij} \in \M_{ij}, ~i \neq j\right\}.$ By a direct calculation we have that if $r_{ii} \in \Z(\R_i)$ and $r_{ii}m_{ij} = m_{ij}r_{jj}$ for every $m_{ij} \in \M_{ij}$ for all $ i \neq j $, then $ \bigoplus_{i=1}^{n} r_{ii} \in \Z(\G)$; that is, $ \left( \bigoplus_{i=1}^{n} \Z(\R_i) \right)\cap S \subseteq \Z(\G).$ To prove that $S = \Z(\G),$ we must show that $\Z(\G) \subseteq S$ and $S \subseteq  \bigoplus_{i=1}^{n} \Z(\R_i).$\\
Suppose that
$x = \left(\begin{array}{cccc}
r_{11} & m_{12} & \ldots & m_{1n}\\
m_{21} & r_{22} & \ldots & m_{2n}\\
 \vdots& \vdots & \ddots & \vdots\\
m_{n1} & m_{n2} & \ddots & r_{nn}\\
\end{array}\right) \in \Z(\G).$
Since $x\big( \bigoplus_{i=1}^{n} a_{ii}\big) = \big( \bigoplus_{i=1}^{n} a_{ii}\big)x$ for all $a_{ii} \in \R_{i},$  we have $a_{ii}m_{ij} = m_{ij}a_{jj}$ for $i \neq j$. Making $a_{jj} = 0$ we conclude $a_{ii}m_{ij} = 0$ for all $a_{ii} \in \R_{i}$  and so $m_{ij} = 0$ for all $i \neq j$ which implies that $x= \bigoplus_{i=1}^{n} r_{ii}$. Moreover, for any $m_{ij} \in \M_{ij}$ as
 $$x \left(\begin{array}{cccccccc}
0  & \ldots & 0 & \ldots & 0 & \cdots & 0\\
 \vdots & \ddots & \vdots & & \vdots & & \vdots\\
 0 & \ldots & 0 & \ldots & m_{ij}& \ldots & 0\\
\vdots & &\vdots & \ddots & \vdots & & \vdots\\
0 &\ldots & 0&\ldots & 0 & \ldots & 0 \\
\vdots & &\vdots & & \vdots & \ddots & \vdots \\
0 & \ldots & 0 & \ldots & 0 & \ldots &  0
\end{array}\right)
=\left(\begin{array}{cccccccc}
0  & \ldots & 0 & \ldots & 0 & \cdots & 0\\
 \vdots & \ddots & \vdots & & \vdots & & \vdots\\
 0 & \ldots & 0 & \ldots & m_{ij}& \ldots & 0\\
\vdots & &\vdots & \ddots & \vdots & & \vdots\\
0 &\ldots & 0&\ldots & 0 & \ldots & 0 \\
\vdots & &\vdots & & \vdots & \ddots & \vdots \\
0 & \ldots & 0 & \ldots & 0 & \ldots &  0
\end{array}\right)x,$$
then $r_{ii}m_{ij} = m_{ij}r_{jj}$ for all $i \neq j$ which results in $\Z(\G) \subseteq S$. Now suppose $ x=\bigoplus_{i=1}^{n} r_{ii} \in S.$ Then for any $a_{ii} \in \R_i$ $(i=1, \cdots ,n-1),$ we have $(r_{ii}a_{ii} - a_{ii}r_{ii})m_{ij} = r_{ii}(a_{ii}m_{ij}) - a_{ii}(r_{ii}m_{ij}) = (a_{ii}m_{ij})r_{jj} - a_{ii}(m_{ij}r_{jj}) = 0$ for all $m_{ij} \in \M_{ij}$  $(i \neq j)$ and hence $r_{ii}a_{ii} - a_{ii}r_{ii} = 0$ as $\M_{ij}$ is left faithful $\R_i$-module.
Now for $i = n$ we have $m_{in}(r_{nn}a_{nn} - a_{nn}r_{nn}) = m_{in}(r_{nn}a_{nn}) - m_{in}(a_{nn}r_{nn}) =(m_{in}r_{nn})a_{nn} - (m_{in}a_{nn})r_{nn}= (r_{ii}m_{in})a_{nn} - r_{ii}(m_{in}a_{nn}) = 0$ and hence $r_{nn}a_{nn} - a_{nn}r_{nn} = 0$ as $\M_{in}$ is right faithful $\R_n$-module. Therefore $r_{ii} \in \Z(\R_i),$ $i = 1, \cdots, n$. Hence, $ S \subseteq \bigoplus_{i=1}^{n} \Z(\R_i).$
\par The fact that $\pi_{\R_i}(\Z(\G)) \subseteq \Z(\R_i)$ for $i = 1 , \cdots , n$ are direct consequences of $ \Z(\G) = S\subseteq \bigoplus_{i=1}^{n} \Z(\R_i).$ Now we prove the existence of the ring isomorphism $\tau^j_i : \pi_{\R_i}(\Z(\G)) \longrightarrow \pi_{\R_j}(\Z(\G))$ for $i \neq j$. For this, let us consider a pair of indices $(i, j)$ such that $ i \neq j$. For any $ r=\bigoplus_{k=1}^{n} r_{kk} \in \Z(\G)$ let us define  $\tau ^j_i(r_{ii})=r_{jj}$. The application is well defined because if $s= \bigoplus_{k=1}^{n} s_{kk} \in \Z(\G)$ is such that $s_{ii} = r_{ii}$, then we have $m_{ij}r_{jj} = r_{ii}m_{ij} = s_{ii}m_{ij}=m_{ij}s_{jj}$ for all $m_{ij} \in \M_{ij}$. Since $\M_{ij}$ is right faithful $\R_j$-module, we conclude that $r_{jj} = s_{jj}$. Therefore, for any $r_{ii} \in \pi_{\R_i}(\Z(\G)),$ there exists a unique $r_{jj} \in \pi_{\R_j}(\Z(\G)),$ denoted by $\tau ^j_i(r_{ii})$. It is easy to see that $\tau^j_i$ is bijective. Moreover, for any $r_{ii}, s
_{ii} \in \pi_{\R_i}(\Z(\G))$ we have
$m_{ij}\tau ^j_i(r_{ii} + s_{ii})=(r_{ii} + s_{ii})m_{ij} =m_{ij}(r_{jj} + s_{jj})=m_{ij}\big(\tau^j_i(r_{ii}) + \tau^j_i(s_{ii})\big)$ and $m_{ij}\tau^j_i(r_{ii}s_{ii}) = (r_{ii}s_{ii})m_{ij} = r_{ii}(s_{ii}m_{ij}) = (s_{ii}m_{ij})\tau^j_i(r_{ii}) = s_{ii}\big(m_{ij}\tau^j_i(r_{ii})\big) = m_{ij}\big( \tau^j_i(r_{ii})\tau^j_i(s_{ii})\big)$.
Thus $\tau^j_i(r_{ii} + s_{ii}) = \tau^j_i(r_{ii}) + \tau^j_i(s_{ii})$ and $\tau^j_i(r_{ii}s_{ii}) = \tau^j_i(r_{ii})\tau^j_i (s_{ii})$ and so $\tau^j_i$ is a ring isomorphism.
 \end{proof}

\begin{prop}\label{ter}
Let $\G$ be a generalized $n-$matrix ring and $ i \neq j$ such that:
\begin{enumerate}
	\item[\it (i)] $a_{ii}\R_i = 0$ implies $a_{ii} = 0$ for $a_{ii} \in \R_i$;
	\item[\it (ii)] $\R_j b_{jj} = 0$ implies $b_{jj} = 0$ for all $b_{jj} \in \R_j$.
\end{enumerate}
Then $u \G = 0$ or $\G u = 0$ implies $u =0$ for $u \in \G$.
\end{prop}
\begin{proof} First, let us observe that if $i \neq j$ and $\R_i a_{ii} = 0,$ then we have $\R_i a_{ii}m_{ij}\R_{j} = 0$, for all $m_{ij} \in  \M_{ij}$, which implies $a_{ii}m_{ij} = 0$ by condition {\it (ii)} of the Definition \ref{pri}. It follows that $a_{ii}\M_{ij} = 0$ resulting in $a_{ii} = 0$. Hence, suppose $ u = \bigoplus_{i, j \in \left\{1, \ldots, n \right\}} u_{ij}$, with $u_{ij} \in \G_{ij}$, satisfying $u\G = 0$. Then $u_{kk}\R_k = 0$ which yields $u_{kk} = 0$ for $k = 1, \cdots, n-1$, by condition {\it (i)}. Now for $k = n$, $u_{nn}\R_n = 0,$ we have $\R_{i}m_{in}u_{nn}\R_{n}= 0$, for all $m_{in} \in  \M_{in}$, which implies $m_{in}u_{nn} = 0$ by condition {\it (ii)} of the Definition \ref{pri}. It follows that $\M_{in}u_{nn} = 0$ which implies $u_{nn} = 0$. Thus $u_{ij}\R_j = 0$ and then $u_{ij} = 0$ by condition {\it (ii)}  of the Definition \ref{pri}. Therefore $u = 0$. Similarly, we prove that if $\G u = 0$ then $u=0$.
\end{proof}

\section{The Main Theorem}

Follows our main result which has the purpose of generalizing Theorem $2.1$ in \cite{Ferreira}.
Our main result reads as follows.

\begin{thm}\label{t11} Let $B : \G \times \G \longrightarrow \G$ be a biadditive map such that:
\begin{enumerate}
\item[{\it (i)}] $B(\G_{pp},\G_{qq})\subseteq \G_{pp}\cap \G_{qq}$; $B(\G_{pp},\G_{rs})\in \G_{rs}$ and $B(\G_{rs},\G_{pp})\in \G_{rs}$;
 $B(\G_{pq},\G_{rs})=0$;
\item[{\it (ii)}] if $B(\bigoplus_{1\leq p\neq q\leq n} c_{pq}, \G_{nn}) = 0$ or $B(\bigoplus _{1\leq r<n} \G_{rr},\bigoplus_{1\leq p\neq q\leq n} c_{pq}) = 0$, then $\bigoplus_{1\leq p\neq q\leq n} c_{pq} = 0$;
\item[{\it (iii)}] $B(\G_{nn}, a_{nn}) = 0$ implies $a_{nn} = 0$;
\item[{\it (iv)}] if $B(\bigoplus_{p=1}^{n} c_{pp},\G_{rs}) = B(\G_{rs},\bigoplus_{p=1}^{n} c_{pp}) = 0$ for all $1\leq r\neq s\leq n$, then $\bigoplus_{p=1}^{n-1} c_{pp} \oplus (-c_{nn}) \in \Z(\G)$;
\item[{\it (v)}] $B(c_{pp},d_{pp}) = B(d_{pp},c_{pp})$ and $B(c_{pp},d_{pp})d_{pn}d_{nn} = d_{pp}d_{pn}B(c_{nn},d_{nn})$ for all $c=\bigoplus_{p=1}^{n} c_{pp} \in \Z(\G)$;
\item[{\it (vi)}] $B\big(c_{rr},B(c_{kl},c_{nn})\big) = B\big(B(c_{rr},c_{kl}), c_{nn}\big)$.
\end{enumerate}
Suppose $f: \G \times \G \longrightarrow \G$ a map satisfying the following conditions:
\begin{enumerate}
\item[\it (vii)] $f(\G,0) = f(0,\G) = 0$;
\item[\it (viii)] $B\big(f(x,y),z\big) = f\big(B(x,z),B(y,z)\big)$;
\item[\it (ix)] $B\big(z,f(x,y)\big) = f\big(B(z,x),B(z,y)\big)$
\end{enumerate}
for all $x,y,z \in \G$. Then $f = 0$.
\end{thm}
\begin{proof} Following the ideas of Ferreira in \cite{Ferreira} we divide the proof into the four cases. Then, let us consider arbitrary elements $x_{kl}, u_{kl}, a_{kl} \in \G_{kl}$ $( k, l \in \left\{1, \ldots, n\right\})$.\\

\noindent \textit{First case.} In this first case the reader should keep in mind that we want to show $$f \big(\sum _{1\leq i< n} x_{ii}, \sum _{1\leq j\neq k\leq n} x_{jk}\big)=0.$$
From the hypotheses of the theorem, we have
\begin{eqnarray*}
B\left(f \big(\sum _{1\leq i< n} x_{ii}, \linebreak \sum _{1\leq j\neq k\leq n} x_{jk}\big),a_{nn}\right)
&=& f\big(B\left(\sum _{1\leq i< n} x_{ii}, a_{nn}\right), B\left(\sum _{1\leq j\neq k\leq n} x_{jk}, a_{nn}\right)\big) \\
&=&  f\big(0, B\left(\sum _{1\leq j\neq k\leq n} x_{jk}, a_{nn}\right)\big) \\&=& 0.
\end{eqnarray*}
Now by condition $(i)$, this implies that
$$ B\left(\sum _{1\leq p, q\leq n} f (\sum _{1\leq i< n} x_{ii}, \sum _{1\leq j\neq k\leq n} x_{jk})_{pq}, a_{nn}\right)=0.$$
Since $$\displaystyle B\left(\sum _{1\leq p< n}f(\sum _{1\leq i< n} x_{ii}, \sum _{1\leq j\neq k\leq n} x_{jk})_{pp}, a_{nn}\right) = 0,$$
$$\displaystyle B\left(\sum _{1\leq p\neq q\leq n} \linebreak f (\sum _{1\leq i< n} x_{ii}, \sum _{1\leq j\neq k\leq n} x_{jk})_{pq}, a_{nn}\right)\in \bigoplus _{1\leq p\neq q\leq n}\G_{pq}$$
 and
$$  B\left(f (\sum _{1\leq i< n} x_{ii},\linebreak \sum _{1\leq j\neq k\leq n} x_{jk})_{nn}, a_{nn}\right) \in \G_{nn},$$
then
$$ \sum _{1\leq p\neq q\leq n}f(\sum _{1\leq i< n} x_{ii}, \sum _{1\leq j\neq k\leq n} x_{jk})_{pq} = 0\mbox{~by~ condition ~(ii)~}.$$
 Next, we have
\begin{eqnarray*}
B\left(a_{nn}, f (\sum _{1\leq i< n} x_{ii}, \sum _{1\leq j\neq k\leq n} x_{jk})\right)
&=& f\left(B(a_{nn}, \sum _{1\leq i< n} x_{ii}), B(a_{nn}, \sum _{1\leq j \neq k\leq n} x_{jk})\right)\\
&=& f\left(0, B(a_{nn}, \sum _{1\leq j\neq k\leq n} x_{jk})\right) \\&=& 0
\end{eqnarray*}
which implies
$$\sum _{1\leq p, q\leq n} B\left(a_{nn},f (\sum _{1\leq i< n} x_{ii}, \sum _{1\leq j\neq k\leq n} x_{jk})_{pq}\right)=0.$$
It follows that $$B\left(a_{nn}, \sum _{1\leq p< n}f(\sum _{1\leq i< n} x_{ii}, \sum _{1\leq j\neq k\leq n} x_{jk})_{pp}\right) = 0,$$
$$B\left(a_{nn}, \sum _{1\leq p\neq q\leq n}f (\sum _{1\leq i< n} x_{ii}, \sum _{1\leq j\neq k\leq n} x_{jk})_{pq}\right)\in \bigoplus _{1\leq p\neq q\leq n}\G_{pq}$$
and
$$B\left(a_{nn}, f (\sum _{1\leq i< n} x_{ii}, \sum _{1\leq j\neq k\leq n} x_{jk})_{nn}\right) \in \G_{nn}.$$
Hence,
$$B\left(a_{nn}, f (\sum _{1\leq i< n} x_{ii},\linebreak \sum _{1\leq j\neq k\leq n} x_{jk})_{nn}\right)=0$$
which yields
$$ f (\sum _{1\leq i< n} x_{ii}, \sum _{1\leq j\neq k\leq n} x_{jk})_{nn} = 0$$
by condition $(iii)$. Yet, we have
\begin{eqnarray*}
B\left(\sum _{1\leq p<n} f(\sum _{1\leq i< n} x_{ii}, \sum _{1\leq j\neq k\leq n} x_{jk})_{pp}, a_{rs}\right)
&=& B\left(f (\sum _{1\leq i< n} x_{ii}, \sum _{1\leq j\neq k\leq n} x_{jk}), a_{rs}\right) \\
&=& f\left(B(\sum _{1\leq i< n} x_{ii}, a_{rs}), B(\sum _{1\leq j\neq k\leq n} x_{jk}, a_{rs})\right)\\
&=&f\left(B(\sum _{1\leq i< n} x_{ii}, a_{rs}),0\right)\\
&=&0
\end{eqnarray*}
and
\begin{eqnarray*}
B\left(a_{rs}, \sum _{1\leq p<n} f (\sum _{1\leq i< n} x_{ii}, \sum _{1\leq j\neq k\leq n} x_{jk})_{pp}\right)
&=& B\left(a_{rs}, f (\sum _{1\leq i< n} x_{ii}, \sum _{1\leq j\neq k\leq n} x_{jk})\right) \\
&=& f\left(B(a_{rs}, \sum _{1\leq i< n} x_{ii}), B(a_{rs}, \sum _{1\leq j\neq k\leq n} x_{jk}) \right)\\
&=& f\left(B(a_{rs}, \sum _{1\leq i< n} x_{ii}), 0 \right) \\
&=& 0.
\end{eqnarray*}
It follows that $\displaystyle \sum _{1\leq p<n} f (\sum _{1\leq i< n} x_{ii}, \sum _{1\leq j\neq k\leq n} x_{jk})_{pp}+ 0 \in \Z(\G)$ and so
$$ \sum _{1\leq p<n} f (\sum _{1\leq i< n} x_{ii}, \sum _{1\leq j\neq k\leq n} x_{jk})_{pp} = 0$$
by Proposition \ref{seg}.
Consequently, we have
$ f \left(\sum _{1\leq i< n} x_{ii}, \sum _{1\leq j\neq k\leq n} x_{jk}\right)=0.$
\\

\noindent \textit{Second case. }
In the second case it must be borne in mind that we want to show $$f (\sum _{1\leq i\neq j\leq n} x_{ij},\sum _{1\leq k\neq l\leq n} y_{kl})=0.$$
From the hypotheses of the theorem ,we have
\begin{eqnarray*}
B\left(\sum _{1\leq p, q\leq n} f (\sum _{1\leq i\neq j\leq n} x_{ij}, \sum _{1\leq k\neq l\leq n} y_{kl})_{pq}, a_{rs}\right)
&=&B\left(f (\sum _{1\leq i\neq j\leq n} x_{ij}, \sum _{1\leq k\neq l\leq n} y_{kl}), a_{rs}\right)\\
&=& f \left(B(\sum _{1\leq i\neq j\leq n} x_{ij}, a_{rs}), B(\sum _{1\leq k\neq l\leq n} y_{kl}, a_{rs})\right)\\
&=& f(0,0)\\
&=&0.
\end{eqnarray*}
Since $\displaystyle B\left(\sum _{1\leq p\neq q\leq n}f (\sum _{1\leq i\neq j\leq n} x_{ij}, \sum _{1\leq k\neq l\leq n} y_{kl})_{pq}, a_{rs}\right)=0$, then
$$
\centerline{$\displaystyle B\big(\sum _{1\leq p\leq n} f (\sum _{1\leq i\neq j\leq n} x_{ij}, \sum _{1\leq k\neq l\leq n} y_{kl})_{pp}, a_{rs}\big)=0$.}
$$
Smilarly, we prove that
$$ B\left(a_{rs},\sum _{1\leq p\leq n} f (\sum _{1\leq i\neq j\leq n} x_{ij}, \sum _{1\leq k\neq l\leq n} y_{kl})_{pp}\right)=0.$$
By condition $(iv),$ it follows that
\begin{eqnarray}\label{centro}
&&\sum _{1\leq p< n} f (\sum _{1\leq i\neq j\leq n} x_{ij}, \sum _{1\leq k\neq l\leq n} y_{kl})_{pp} + \left(- f(\sum _{1\leq i\neq j\leq n}x_{ij}, \sum _{1\leq k\neq l\leq n} y_{kl})_{nn}\right) \in \Z(\G).
\end{eqnarray}
 Now, we observe that
\begin{eqnarray*}
B\left(f (\sum _{1\leq i\neq j\leq n} x_{ij},\sum _{1\leq k\neq l\leq n} y_{kl}), a_{nn}\right)
&=&f\big(B\left(\sum _{1\leq i\neq j\leq n} x_{ij},a_{nn}\right),B\left(\sum _{1\leq k\neq l\leq n} y_{kl},a_{nn}\right)\big)\\
&=&f\big(\sum _{1\leq i\neq j\leq n} B(x_{ij},a_{nn}),\sum _{1\leq k\neq l\leq n}B( y_{kl},a_{nn})\big).
\end{eqnarray*}
With (\ref{centro}), this implies that
$$ \sum _{1\leq p< n} B\left(f (\sum _{1\leq i\neq j\leq n}\linebreak x_{ij}, \sum _{1\leq k\neq l\leq n} y_{kl}),a_{nn}\right)_{pp}+\big(- B\left(f (\sum _{1\leq i\neq j\leq n} x_{ij},\sum _{1\leq k\neq l\leq n} y_{kl}),a_{nn}\right)_{nn}\big)\in \Z(\G).$$
Since $\displaystyle B\left(f (\sum _{1\leq i\neq j\leq n} x_{ij}, \sum _{1\leq k\neq l\leq n} y_{kl}), a_{nn}\right)\in \bigoplus _{1\leq p\neq q\leq n} \G_{pq}\bigoplus \G_{nn}$ then\\
 $\displaystyle \sum _{1\leq p< n} B\left(f (\sum _{1\leq i\neq j\leq n} x_{ij}, \sum _{1\leq k\neq l\leq n}y_{kl}),a_{nn}\right)_{pp}=0$ which results in
$$\displaystyle B\left(f (\sum _{1\leq i\neq j\leq n} x_{ij}, \sum _{1\leq k\neq l\leq n} y_{kl}),a_{nn}\right)_{nn}=0 \mbox{~by ~Proposition ~\ref{seg}}.$$
 Hence $\displaystyle B\left(f (\sum _{1\leq i\neq j\leq n} x_{ij}, \sum _{1\leq k\neq l\leq n} y_{kl}),a_{nn}\right)\in \bigoplus _{1\leq p\neq q\leq n} \G_{pq}$. It follows that
\begin{eqnarray*}
\lefteqn{B\left(a_{rr}, B\left(f (\sum _{1\leq i\neq j\leq n} x_{ij}, \sum _{1\leq k\neq l\leq n} y_{kl}), a_{nn}\right)\right)}\\
&=&B\left(a_{rr},f \big(B\left(\sum _{1\leq i\neq j\leq n} x_{ij}, a_{nn}\right), B\left(\sum _{1\leq k\neq l\leq n} y_{kl}, a_{nn}\right)\big)\right) \\
&=&f \big(B\left(a_{rr}, B\left(\sum _{1\leq i\neq j\leq n} x_{ij}, a_{nn}\right)\right), B\left(a_{rr}, B\left(\sum _{1\leq k\neq l\leq n} y_{kl}, a_{nn}\right)\right)\big) \\
&=&f \big(B\left(a_{rr},  B\left(\sum _{1\leq i\neq j\leq n} x_{ij}, a_{nn}\right)\right), B\left(B\left(a_{rr}, \sum _{1\leq k\neq l\leq n} y_{kl}\right), a_{nn}\right)\big) \\
&=&f\big(B\left(a_{rr}, a_{nn} +  B\left(\sum _{1\leq i\neq j\leq n} x_{ij}, a_{nn}\right)\right), B\left(B\left(a_{rr}, \sum _{1\leq k\neq l\leq n} y_{kl}\right), a_{nn} + B\left(\sum _{1\leq i\neq j\leq n} x_{ij}, a_{nn}\right)\right)\big)\\
&=&B\left(f\big(a_{rr}, \linebreak B\left(a_{rr}, \sum _{1\leq k\neq l\leq n} y_{kl}\right)\big), a_{nn}
 + B\left(\sum _{1\leq i\neq j\leq n} x_{ij}, a_{nn}\right)\right)  \\
 &=& B\left(0,  a_{nn} +B\left(\sum _{1\leq i\neq j\leq n} x_{ij},  a_{nn}\right)\right)\\
 &=&0
\end{eqnarray*}
by first case, for all $1\leq r<n$.
\par
So $\displaystyle B\left(f (\sum _{1\leq i\neq j\leq n}  x_{ij}, \sum _{1\leq k\neq l\leq n} y_{kl}), a_{nn}\right)= 0$, by condition $(ii)$. It follows that
\begin{eqnarray*}
\lefteqn{\sum _{1\leq p\leq n} B\left(f (\sum _{1\leq i\neq j\leq n} x_{ij}, \sum _{1\leq k\neq l\leq n} y_{kl})_{pp}, a_{nn}\right)}\\
&&+\sum _{1\leq p\neq q\leq n}  B\left(f (\sum _{1\leq i\neq j\leq n} x_{ij}, \sum _{1\leq k\neq l\leq n} y_{kl})_{pq}, a_{nn}\right)=0
\end{eqnarray*}
which yields $$B\left(\sum _{1\leq p\neq q\leq n}f (\sum _{1\leq i\neq j\leq n} x_{ij}, \sum _{1\leq k\neq l\leq n} y_{kl})_{pq}, a_{nn}\right)=0$$ and so
$$\sum _{1\leq p\neq q\leq n}f (\sum _{1\leq i\neq j\leq n} x_{ij}, \sum _{1\leq k\neq l\leq n} y_{kl})_{pq}=0 \mbox{ ~by~ condition ~(ii).}
$$
Hence,
\begin{eqnarray*}
B\left(a_{nn},f (\sum _{1\leq i\neq j\leq n} x_{ij}, \sum _{1\leq k\neq l\leq n} y_{kl})_{nn}\right)
&=&B\left(a_{nn},f (\sum _{1\leq i\neq j\leq n} x_{ij}, \sum _{1\leq k\neq l\leq n} y_{kl})\right)\\
&=& f\big(B\left(a_{nn},\sum _{1\leq i\neq j\leq n} x_{ij}\right),B\left(a_{nn},\sum _{1\leq k\neq l\leq n} y_{kl}\right)\big)
\end{eqnarray*}
and by (\ref{centro}) above we have
\begin{eqnarray*}
\lefteqn{\sum _{1\leq p<n} B\left(a_{nn},f (\sum _{1\leq i\neq j\leq n} x_{ij}, \sum _{1\leq k\neq l\leq n} y_{kl})_{nn}\right)_{pp}}\\
&& +\big(-B\left(a_{nn},f (\sum _{1\leq i\neq j\leq n} x_{ij}, \sum _{1\leq k\neq l\leq n} y_{kl})_{nn}\right)_{nn}\big)\in \Z(\G).
\end{eqnarray*}
Since $$B\left(a_{nn},f (\sum _{1\leq i\neq j\leq n} x_{ij}, \sum _{1\leq k\neq l\leq n} y_{kl})_{nn}\right)\in \G_{nn}$$ then we have
$$\sum _{1\leq p<n} B\left(a_{nn},f (\sum _{1\leq i\neq j\leq n} x_{ij}, \sum _{1\leq k\neq l\leq n} y_{kl})_{nn}\right)_{pp}=0$$ and so
\begin{eqnarray*}
B\left(a_{nn},f (\sum _{1\leq i\neq j\leq n} x_{ij}, \sum _{1\leq k\neq l\leq n} y_{kl})_{nn}\right)&=&B\left(a_{nn},f (\sum _{1\leq i\neq j\leq n} x_{ij},\sum _{1\leq k\neq l\leq n} y_{kl})_{nn}\right)_{nn}=0,
\end{eqnarray*}
by Proposition \ref{seg}. It follows that $\displaystyle f (\sum _{1\leq i\neq j\leq n} x_{ij}, \sum _{1\leq k\neq l\leq n} y_{kl})_{nn}=0$, by condition  $(iii)$, which implies $$\displaystyle \sum _{1\leq p< n}f(\sum _{1\leq i\neq j\leq n} x_{ij}, \sum _{1\leq k\neq l\leq n} y_{kl})_{pp}=0,$$ by (\ref{centro}). Consequently, we have
$$\displaystyle f (\sum _{1\leq i\neq j\leq n} x_{ij},\sum _{1\leq k\neq l\leq n} y_{kl})=0.$$
\\

\noindent \textit{Third case.}
Here, in the third case, we are interested in checking $$f\big(\sum _{1\leq p< n} x_{pp} + \sum  _{1\leq p\neq q\leq n} x_{pq}, \sum _{1\leq k< n} u_{kk} + \sum _{1\leq k\neq l\leq n} u_{kl}\big)=0.$$
In view of second case, we Observe that
\begin{eqnarray*}
\lefteqn{B\left(f\big(\sum _{1\leq p< n} x_{pp} + \sum  _{1\leq p\neq q\leq n} x_{pq}, \sum _{1\leq k< n} u_{kk} + \sum _{1\leq k\neq l\leq n} u_{kl}\big), a_{rs}\right)}\\
&=&f (B\left(\sum _{1\leq p< n} x_{pp} + \sum  _{1\leq p\neq q\leq n} x_{pq}, a_{rs}\right), B\left(\sum _{1\leq k< n} u_{kk} + \sum _{1\leq k\neq l\leq n} u_{kl}, a_{rs}\right))\\
&=& f\big(\sum  _{1\leq p<n}B(x_{pp}, a_{rs}),\sum  _{1\leq k<n} B(u_{kk}, a_{rs})\big)\\&=& 0.
\end{eqnarray*}
It follows that
$$\sum _{1\leq t\leq n}  B\left(f  \big(\sum _{1\leq p< n} x_{pp} +  \sum  _{1\leq p\neq q\leq n} x_{pq}, \sum _{1\leq k< n} u_{kk} +  \sum _{1\leq k\neq l\leq n} u_{kl}\big)_{tt}, a_{rs}\right)= 0.$$
Similarly, we have
$$\sum _{1\leq t\leq n} B\left(a_{rs},f\big( \sum _{1\leq p< n} x_{pp} + \sum  _{1\leq p\neq q\leq n} x_{pq}, \sum _{1\leq k< n} u_{kk} + \sum _{1\leq k\neq l\leq n} u_{kl}\big)_{tt}\right)= 0.$$
It follows that
\begin{eqnarray*}
\lefteqn{\sum _{1\leq t< n} f\big(\sum _{1\leq p< n} x_{pp} + \sum  _{1\leq p\neq q\leq n} x_{pq}, \sum _{1\leq k< n} u_{kk} + \sum _{1\leq k\neq l\leq n} u_{kl}\big)_{tt}}\\
&&+ \big(- f \big(\sum _{1\leq p< n} x_{pp} + \sum  _{1\leq p\neq q\leq n} x_{pq}, \sum _{1\leq k< n} u_{kk} + \sum _{1\leq k\neq l\leq n} u_{kl}\big)_{nn}\in \Z(\G)
\end{eqnarray*}
by condition $(iv)$. But
\begin{eqnarray*}
\lefteqn{B\left(f\big(\sum _{1\leq p< n} x_{pp} + \sum  _{1\leq p\neq q\leq n} x_{pq}, \sum _{1\leq k< n} u_{kk} + \sum _{1\leq k\neq l\leq n} u_{kl}\big), a_{nn}\right)}\\
&=&f \big(B\left(\sum _{1\leq p< n} x_{pp} + \sum  _{1\leq p\neq q\leq n} x_{pq}, a_{nn}\right), B\left(\sum _{1\leq k< n} u_{kk} + \sum _{1\leq k\neq l\leq n} u_{kl}, a_{nn}\right)\big)\\
&=&f \big(B\left(\sum  _{1\leq p\neq q\leq n} x_{pq}, a_{nn}\right), B\left(\sum _{1\leq k\neq l\leq n} u_{kl}, a_{nn}\right)\big)\\
&=&f \big(\sum_{1\leq p\neq q\leq n} B\left(x_{pq}, a_{nn}\right), \sum _{1\leq k\neq l\leq n} B\left(u_{kl}, a_{nn}\right)\big)\\&=&0
\end{eqnarray*}
by second case. As a  result, we have
$$\displaystyle \sum _{1\leq r\neq s\leq n}f\big(\sum _{1\leq p< n} x_{pp} + \sum  _{1\leq p\neq q\leq n} x_{pq}, \sum _{1\leq k< n} u_{kk} + \sum _{1\leq k\neq l\leq n} u_{kl}\big)_{rs}=0 \mbox{~by~ condition ~(ii).}$$
Hence from the second case
\begin{eqnarray*}
\lefteqn{B\left(a_{nn},f\big(\sum _{1\leq p< n} x_{pp} + \sum  _{1\leq p\neq q\leq n} x_{pq}, \sum _{1\leq k< n} u_{kk} + \sum _{1\leq k\neq l\leq n} u_{kl}\big)\right)}\\
&=&f \big(B\left(a_{nn},\sum _{1\leq p< n} x_{pp} + \sum  _{1\leq p\neq q\leq n} x_{pq}\right), B\left(a_{nn},\sum _{1\leq k< n} u_{kk} + \sum _{1\leq k\neq l\leq n} u_{kl}\right)\big)\\
&=&f \big(B\left(a_{nn},\sum  _{1\leq p\neq q\leq n} x_{pq}\right), B\left(a_{nn},\sum _{1\leq k\neq l\leq n} u_{kl}\right)\big)\\
&=&f \big(\sum  _{1\leq p\neq q\leq n} B\big(a_{nn},x_{pq}\big), \sum _{1\leq k\neq l\leq n} B\big(a_{nn},u_{kl}\big)\big)\\&=&0.
\end{eqnarray*}
This implies
$$\displaystyle B\left(a_{nn},f\big(\sum _{1\leq p< n} x_{pp} + \sum  _{1\leq p\neq q\leq n} x_{pq}, \sum _{1\leq k< n} u_{kk} + \sum _{1\leq k\neq l\leq n} u_{kl}\big)_{nn}\right)=0.$$
Thus
$$\displaystyle f\big(\sum _{1\leq p< n} x_{pp} + \sum  _{1\leq p\neq q\leq n} x_{pq}, \sum _{1\leq k< n} u_{kk} + \sum _{1\leq k\neq l\leq n} u_{kl}\big)_{nn}=0$$
implying
$$\displaystyle \sum _{1\leq t< n} f\big(\sum _{1\leq p< n} x_{pp} + \sum  _{1\leq p\neq q\leq n} x_{pq}, \sum _{1\leq k< n} u_{kk} + \sum _{1\leq k\neq l\leq n} u_{kl}\big)_{tt}=0$$
by condition $(iii)$.
Therefore,
$$\displaystyle f\big(\sum _{1\leq p< n} x_{pp} + \sum  _{1\leq p\neq q\leq n} x_{pq}, \sum _{1\leq k< n} u_{kk} + \sum _{1\leq k\neq l\leq n} u_{kl}\big)=0.$$\\

\noindent \textit{Fourth case.}
Finally in the last case we show that $f = 0$.\\
Since $\displaystyle B\left(\sum_{1 \leq p , q \leq n}x_{pq}, y_{rs}\right) \subseteq \G_{rs}$ we have
$ B(f (x, u), a_{rs}) = f (B(x, a_{rs}), B(u, a_{rs})) = 0.$
 Then by second case, we obtain
$$\displaystyle B\left(\sum_{1\leq p \leq n}f (x, u)_{pp} , a_{rs}\right) = 0.$$
Similarly, we have
$$\displaystyle B\left(a_{rs}, \sum_{1\leq p \leq n}f (x, u)_{pp}\right) = 0.$$
It follows from condition $(iv)$ that
$\displaystyle \sum_{1\leq p< n}f (x, u)_{pp}+(-f(x,u)_{nn}) \in \Z(\G)$.
\par Now as $\displaystyle B\left(\sum_{1 \leq r< n}y_{rr}, y\right) \subseteq \sum_{1 \leq r< n}\G_{rr}+\sum_{1 \leq r\neq s \leq n}\G_{rs}$ then  by third case, we have
$$\displaystyle B\left(\sum_{1 \leq r< n}a_{rr},f (x, u)\right) = f \big(B\left(\sum_{1 \leq r< n}a_{rr},x\right), B\left(\sum_{1 \leq r< n}a_{rr},u\right)\big) = 0.$$
 It follows that
$\displaystyle B\left(\sum_{1 \leq r< n}a_{rr},\sum_{1 \leq r< n}f (x, u)_{rr}+\sum_{1 \leq r\neq s\leq n}f(x,u)_{rs}\right)=0$ implying
\begin{enumerate}
\item[ (1)] $\displaystyle B\left(\sum_{1 \leq r< n}a_{rr},\sum_{1 \leq r< n}f (x, u)_{rr}\right)=0$,
\item[ (2)] $\displaystyle B\left(\sum_{1 \leq r< n}a_{rr},\sum_{1 \leq r\neq s\leq n}f(x,u)_{rs}\right)=0$.
\end{enumerate}
By identity $(1)$ above we have
$\displaystyle \sum_{1 \leq r< n}B\big(a_{rr},f (x, u)_{rr}\big)=0$ resulting $B\big(a_{rr},f (x, u)_{rr}\big)=0$ for all $1 \leq r< n$. We deduce
\begin{eqnarray*}
0&=&B\big(a_{rr},f (x, u)_{rr}\big)a_{rn}a_{nn}\\
&=&B\big(f (x, u)_{rr},a_{rr}\big)a_{rn}a_{nn}\\
&=&a_{rr}a_{rn}B\big(-f(x, u)_{nn},a_{nn}\big)\\
&=&a_{rr}a_{rn}B\big(a_{nn},-f(x, u)_{nn}\big)
\end{eqnarray*}
for all $r<n$, by condition $(v)$. It follows that $B\big(a_{nn},f(x, u)_{nn}\big)=0$ which implies $f(x, u)_{nn}=0$, by condition $(iii)$. Thus, we have $\displaystyle \sum_{1\leq p< n}f (x, u)_{pp}=0$. Now, by identity $(2)$, we have $\displaystyle \sum_{1 \leq r\neq s\leq n}f(x,u)_{rs}=0$ by condition $(ii)$. Hence, we conclude that $f=0$.
\end{proof}
\begin{cor}\label{util}
Let $\G$ be a generalized $n-$matrix ring such that
\begin{enumerate}
\item[\it (i)] For $a_{ii} \in \R_i$, if $a_{ii}\R_i$ = 0, then $a_{ii} = 0$;
\item[\it (ii)] For $b_{jj} \in \R_j,$ if $\R_j b_{jj} = 0,$ then $b_{jj} = 0,$
\end{enumerate}
where $1 \leq i \neq j\leq n.$
Let $k$ be a positive integer. If a map $f : \G \times \G \longrightarrow \G$ satisfies
\begin{enumerate}
\item[\it (i)] $f (\G, 0) = f (0, \G) = 0;$
\item[\it (ii)] $f (x, y)z_1z_2 \cdots z_k = f (xz_1z_2 \cdots z_k, yz_1z_2 \cdots z_k);$
\item[\it (iii)] $z_1z_2 \cdots z_kf (x, y) = f (z_1z_2 \cdots z_k x, z_1z_2 \cdots z_k y),$
\end{enumerate}
for all $x, y, z_1, z_2, \cdots , z_k \in \G,$ then $f = 0.$
\end{cor}
\begin{proof} We first claim that $f (x, y)z = f (xz, yz)$ and $zf (x, y) = f (zx, zy)$ for all $x, y, z \in \G.$
Indeed, since
$$f (x, y)(zz_1)z_2 \cdots z_k = f (xzz_1z_2 \cdots z_k, yzz_1z_2 \cdots z_k) = f (xz, yz)z_1z_2 \cdots z_k,$$
that is, $(f (x, y)z - f (xz, yz))\G^k = 0.$ Hence $f (x, y)z = f (xz, yz)$ by Proposition \ref{ter}. Analogously,
$zf (x, y) = f (zx, zy).$
Define $B : \G \times \G \longrightarrow \G$ by
$B(x, y) = xy.$
It is easy to check that $B$ and $f$ satisfy the all conditions of Theorem \ref{t11}. Hence $f = 0.$
\end{proof}
\section{Applications}

\begin{thm}
Let $\G$ be a generalized $n-$matrix ring such that
\begin{enumerate}
\item[\it (i)] For $a_{ii} \in \R_i$, if $a_{ii}\R_i$ = 0, then $a_{ii} = 0$;
\item[\it (ii)] For $b_{jj} \in \R_j,$ if $\R_j b_{jj} = 0,$ then $b_{jj} = 0,$
\end{enumerate}
where $1 \leq i \neq j\leq n.$
Then every $m-$multiplicative isomorphism from $\G$ onto a ring $\R$ is additive.
\end{thm}
\begin{proof} Suppose that $\varphi$ is a $m-$multiplicative isomorphism from $\G$ onto a ring $\R.$ Since $\varphi$ is onto, $\varphi(x) = 0$ for some $x \in \G.$ Then $\varphi(0) = \varphi(0 \cdots 0x) = \varphi(0) \cdots \varphi(0) \varphi(x) = \varphi(0) \cdots \varphi(0)0 = 0$ and so $\varphi^{-1}(0)
= 0.$ Let us check that the conditions of the Corollary \ref{util} are satisfied. For every $x, y \in \G$ we define $f (x, y) = \varphi^{-1}(\varphi(x + y) - \varphi(x) - \varphi(y)),
$ we see that $f (x, 0) = f (0, x) = 0$ for all $x \in \G.$ It is easy to check that $\varphi^{-1}$ is also a $m$-multiplicative
isomorphism. Thus, for any $u_1, \cdots , u_{m-1} \in \G,$ we have
\begin{eqnarray*}
f (x, y)u_1 \cdots u_{m-1}&=& \varphi^{-1}(\varphi(x + y) - \varphi(x) - \varphi(y) )
\varphi^{-1}(\varphi(u_1)) \cdots \varphi^{-1}(\varphi(u_{m-1}))\\&=& \varphi^{-1}((\varphi(x + y) - \varphi(x) - \varphi(y))\varphi(u_1) \cdots \varphi(u_{m-1})) \\&=& f (xu_1 \cdots u_{m-1}, yu_1 \cdots u_{m-1}).
\end{eqnarray*}
Similarly we have
$u_1 \cdots u_{m-1}f (x, y) = f (u_1 \cdots u_{m-1}x, u_1 \cdots u_{m-1}y)$. Therefore by Corollary \ref{util}, $f = 0.$ That is, $\varphi(x + y) = \varphi(x) + \varphi(y)$ for all $x, y \in \G.$
\end{proof}
\begin{thm}
Let $\G$ be a generalized $n-$matrix ring such that
\begin{enumerate}
\item[\it (i)] For $a_{ii} \in \R_i$, if $a_{ii}\R_i$ = 0, then $a_{ii} = 0$;
\item[\it (ii)] For $b_{jj} \in \R_j,$ if $\R_j b_{jj} = 0,$ then $b_{jj} = 0,$
\end{enumerate}
where $1 \leq i \neq j\leq n.$
Then any $m-$multiplicative derivation d of $\G$ is additive.
\end{thm}
\begin{proof} We define
$f (x, y) = d(x + y) - d(x) - d(y)$, for any $x, y \in \G$. Hence $f$ defined in this way satisfy the conditions of Corollary \ref{util}. Therefore $f = 0$ and so $d(x + y) = d(x) + d(y).$
\end{proof}

\end{document}